%% file: preprint.tex
\documentclass[letterpaper, 10pt, conference]{ieeeconf} % initial submission
\usepackage[noadjust]{cite}
\usepackage{amsmath,amssymb,amsfonts,mathrsfs}
\usepackage{algorithmic}
\usepackage{graphicx}
\usepackage{textcomp}
\usepackage{siunitx}
\def\BibTeX{{\rm B\kern-.05em{\sc i\kern-.025em b}\kern-.08em
    T\kern-.1667em\lower.7ex\hbox{E}\kern-.125emX}}

\usepackage[IEEE]{altthm-styles}
\usepackage{algorithmic-styles}
    
\newcommand\prob[2][]{\operatorname P_{#1}\{#2\}}

\newcommand\rdvar[2][\Omega]{#2^{#1}}

\newenvironment{alignp}{\align}{\endalign}

\IEEEoverridecommandlockouts

\begin{document}

\title{Probabilistic Region-of-Attraction Estimation with Scenario Optimization and Converse Theorems}
\author{Torbjørn Cunis%, \IEEEmembership{Member, IEEE}
\thanks{This research was supported through the Land Baden-Württemberg.}
\thanks{The author is with the University of Stuttgart, 70569 Stuttgart, Germany. (e-mail: {\tt torbjoern.cunis@ifr.uni-stuttgart.de}).}
}

\maketitle

\begin{abstract}
The region of attraction characterizes well-behaved and safe operation of a nonlinear system and is hence sought after for verification.
%For nonlinear systems, estimating the region of attraction is part of a quantitative analysis for safety verification. While analytical methods scale badly for large problems, sampling-based methods are applicable to complex systems even with partially unknown dynamics. 
%Previous results, however, either failed to represent the region of attraction accurately, lacked probabilistic guarantees, or the necessary number of samples too grew with the complexity.
In this paper, a framework for probabilistic region of attraction estimation is developed that combines scenario optimization and converse theorems. With this approach, the probability of an unstable condition being included in the estimate is independent of the system's complexity, while convergence in probability to the true region of attraction is proven.
Numerical examples demonstrate the effectiveness for optimization-based control applications. Combining systems theory and sampling, the complexity of Monte--Carlo-based verification techniques can be reduced.
The results can be extended to arbitrary level sets of which the defining function can be sampled, such as finite-horizon viability.
Thus, the proposed approach is applicable and/or adaptable to verification of a wide range of safety-related properties for nonlinear systems including feedback laws based on optimization or learning. 
\end{abstract}

% final submission
%\begin{IEEEkeywords}
%% alphabetically, from http://www.ieee.org/organizations/pubs/ani\_prod/keywrd98.txt
%Constrained control,
%Optimization,
%Randomized algorithms,
%Stability of nonlinear systems.
%\end{IEEEkeywords}

\section{Introduction}
{Estimating} the region of attraction %and reachable sets 
is a classical problem in the analysis of nonlinear dynamic systems \cite{genesioEtAl1985}. 
The region of attraction of a dynamic system is the set of all initial conditions such that system states asymptotically converge to a given equilibrium, thus describing admissible excitations. For closed-loop nonlinear control systems, the region of attraction is typically bounded due to inaccuracies of the underlying models or physical limitations of the controls. By assessing that the region of attraction is sufficiently large, stable asymptotic behaviour of a nonlinear system can be verified for the envisaged operating envelope.
Recently, region-of-attraction estimation has seen increased interest in the context of feasibility of model-predictive control \cite{Cunis2021aut, Skibik2022tac}, stability of time-distributed optimization \cite{Leung2021b}, and verification of neural networks \cite{Hashemi2020, Yin2021a, Fazlyab2020}.
For the purpose of verification, the region-of-attraction estimate should be nonconservative without overapproximating the true set.

The estimation is often cast as the problem to find a Lyapunov functions subject to a dissipation inequality but analytical solutions (such as sum-of-squares optimization \cite{Hashemi2020, Yin2021a, Fazlyab2020, topcuEtAl2008, chakrabortyEtAl2011, yinEtAl2019}) often require an algebraic (polynomial) approximation of the system dynamics. Furthermore, the complexity of sum-of-squares problems notoriously increases with the number of states considered, a particular issue for methods based on optimization and machine learning with numerous controller states.
An alternative are data-driven methods exploiting converse Lyapunov theorems \cite{Colbert2018, Lai2021scitech}, which approximate a Lyapunov function based on sampled, stable trajectories. These approaches face some major challgenges: First, judging asymptotic behaviour from finite sequences involves guess work; and second, there are no guarantees for the approximation to be a subset of the true region of attraction. Relaxing asymptotic convergence to finite-step reachability of a provably stable subset, scenario optimization can provide probabilistic bounds for the accuracy of the region-of-attraction estimate based on the number of samples \cite{Tempo2013}. Previous work on data-driven reachable set estimation considered probabilistic outer approximations by ellipsoids \cite{Devonport2020}. For the complementary problem of maximal positively invariant set, \cite{Korda2020} proved probabilistic  bounds for sample-based estimates using arbitrary basis functions; however, the number of samples to ensure a given confidence level is difficult to compute and grows exponentially with the dimension of the basis. In \cite{Shen2022}, subsets of the region of attraction were approximated by spheres or polyhedra which only provide conservative estimates.

In this paper, we focus on estimating the region of attraction for a class of linear systems %subject to nonlinear projection operators which are 
commonly arising in optimization schemes and neural networks. We employ a limit on the truncated converse Lyapunov function given by \cite{Balint2006} as sufficient condition for asymptotic stability of a finitely sampled sequence; for the class of systems in this paper, we can find an upper bound as maximum of a monotone scalar function. Solving optimization problems over two independent sets of sampled stable and unstable trajectories, respectively, we aim to find a probabilistic (inner) approximations with a polynomial shape obtained from data. We prove that our approximations converge (in probability and distance) to the true region of attraction if the number of samples and the maximum polynomial degree increases.
Parts of our results can, with minor modifications, also be used for the approximation of reachable sets and/or applied to general, nonlinear or even unknown dynamics.

The remainder of this paper is organized as follows: 
Section~\ref{sec:problem} states the problem of region of attraction estimation via truncated converse Lyapunov functions.
Section~\ref{sec:scenopt} defines optimization over random variables and relates the sample complexity to the reliability of the empirical maximum; Section~\ref{sec:setapprox} introduces the notion of inner and outer approximations for the region of attraction; Section~\ref{sec:probest} presents our main probabilistic algorithm. The algorithm is analyzed and its convergence in probability to the true region of attraction is proven in Section~\ref{sec:analysis}. 
Finally, %Section~\ref{sec:invariance} details the computation of a sufficient upper bound for the truncated converse Lyapunov function and 
numerical examples for saturated LQR and suboptimal model-predictive control are presented in Section~\ref{sec:examples}.

\subsubsection*{Notation}
Let $\mathbb N$ and $\mathbb R$ denote the natural and real numbers, respectively.
We consider the Euclidean space~$\mathbb R^n$ equipped with inner product $\langle \cdot, \cdot \rangle: \mathbb R^n \times \mathbb R^n \to \mathbb R$ and induced norm $\| \cdot \|_2: x \mapsto \sqrt{\langle x, x \rangle}$. 
Let $\mathbb S_m$ (resp., $\mathbb S_m^+$) denote the subset of symmetric (resp., positive semidefinite) matrices in $\mathbb R^{m \times m}$ and take $A, B \in \mathbb R^{m \times \ell}$; the Frobenius inner product is $\langle A, B \rangle = \operatorname{trace} (A^{\mathrm T} B)$; and the Kronecker outer product is $A \otimes B \in \mathbb R^{m\ell \times m\ell}$.
The norm ball of radius $r > 0$, %$\mathscr B^*(r) \subset \mathbb R^n$, 
excluding the origin, is defined as $\mathscr B^*(r) = \{ x \in \mathbb R^n \mid 0 < \| x \|_2 < r \}$. When applied to vectors of same lengths, the inequalities ``$\leq$'' and ``$\geq$'' are to be understood element-wise.

A probability space is a tuple $(\Omega, \Sigma, P)$, where $\Omega$ is the sample space, the events $\Sigma$ are a $\sigma$-algebra of subsets of $\Omega$, and the probability function $P: \Sigma \to [0, 1]$ satisfies that $P(\Omega) = 1$ and $P(S_1 \cup S_2) = P(S_1) + P(S_2)$ for any $S_1, S_2 \subset \Omega$ with $S_1 \cap S_2 = \varnothing$. A random variable defined on a probability space $(\Omega, \Sigma, P)$ is a measurable function $z: \Omega \to \mathcal S$, or $z \in \rdvar {\mathcal S}$ for short, where $\mathcal S$ is a finite-dimensional vector space (the {\em support} of $z$); we write the probability that $z \in A$ for some $A \subset \mathcal S$ as $\prob {z \in A} = P(\{\omega \in \Omega \, | \, z(\omega) \in A \})$.
A sequence of random variables $\{ z_k \in \Omega \to \mathcal S \}_{k \geq 0}$, or short $\{ z_k \}_{k \geq 0} \subset \rdvar {\mathcal S}$, is said to converge {\em almost surely} to the random variable $z \in \rdvar {\mathcal S}$ if and only if
\begin{align*}
	\prob {\lim_{k \to \infty} \| z_k - z \|_2 = 0} = 1
\end{align*}
and to converge {\em in probability} to $z$ if and only if
\begin{align*}
	\lim_{k \to \infty} \prob {\| z_k - z \|_2 < \epsilon} = 1
\end{align*}
for any $\epsilon > 0$.

%To that extent, let $y \in \mathcal D$ be either a deterministic vector or a random variable; %and $f$ be a given scalar function on $\mathcal D$. 
We denote the vector of monomials of $y \in \mathbb R^n$ up to degree $q \in \mathbb N$ by $Z_q\{y\}$, its length by $n_q$, and the outer product as $Z_q^2 = Z_q \otimes Z_q$. A polynomial function $p: \mathcal D \to \mathbb R$ with coefficients $\theta \in \mathbb R^{n_q}$ is given by $p: y \mapsto \langle \theta, Z_q\{y\} \rangle$.

\section{Problem Statement}
\label{sec:problem}
We consider solutions $\{ x_k \}_{k \geq 0} \subset \mathbb R^n$ to the discrete-time initial value problem
\begin{align}
	\label{eq:system}
	\Pi(x): \, \left\{
	\begin{aligned}
	x_{k+1} &= A x_k + \phi(x_k) \\
	x_0 &= x
	\end{aligned}
	\right.
\end{align}
where $\mathbb R^n$ is the state space, $A: \mathbb R^n \to \mathbb R^n$ is a linear operator, and $\phi: \mathbb R^n \to \mathbb R^n$ is a deterministic nonlinear operator. The dynamics $\Pi(x)$ are possibly unknown but bestowed with the following assumption.

\begin{assumption}
	\label{ass:regularity}
    The function $\phi$ is Lipschitz continuous and satisfies $\phi(0) = 0$. Moreover, the origin is locally asymptotically stable for \eqref{eq:system}.
\end{assumption}

Since we assume that \eqref{eq:system} is locally asymptotically stable around the origin, trajectories starting sufficiently close to the origin remain close and there exists a nonempty (but not necessarily connected) set of initial conditions that lead to converging state trajectories.

\begin{definition}
	The true {\em region of attraction} $\mathcal R_\infty$ of \eqref{eq:system} is the largest set $R \subset \mathbb R^n$ such that any solution $\{x_k\}_{k \geq 0}$ of $\Pi(x)$ converges to the origin if $x \in R$.
\end{definition}

We introduce the Hausdorff distance of two bounded sets $A, B \subset \mathbb R^n$ as
\begin{align*}
	\varrho_{\mathrm H}(A, B) = \max \{ \sup_{x \in A} \operatorname{dist}(x; B), \sup_{y \in B} \operatorname{dist}(y; A) \}
\end{align*}
where $\operatorname{dist}$ denotes the distance between a point and a set, i.e., $\operatorname{dist}(x; B) = \inf_{y \in B} \| x - y \|_2$.

\begin{problem}
	Solve
	\begin{align*}
		\min_\theta \varrho_{\mathrm H}(R_\theta, \mathcal R_\infty) \quad \text{s.t. $R_\theta \subset \mathcal R_\infty$}
	\end{align*}
	where $R_\theta \subset \mathbb R^n$ is a set estimate parametrized by $\theta$.
\end{problem}

By Assumption~\ref{ass:regularity}, the origin is a stationary condition and $\phi$ is bound (by $\zeta > 0$), that is, $|| \phi(x) || \leq \zeta || x ||$ holds for all $x \in \mathbb R^n$. Such systems arise in control of linear systems by neural networks, where the state space comprises of system states and hidden layer neurons, and $\phi$ is a nonlinear activation function (such as $\tanh$ or $\max\{x, 0\}$); in linear model-predictive control using a convex solver, possibly with a fixed number of iterations, where $\phi$ is the projection onto the feasible set; or, with some relaxation of the continuity assumptions, in the analysis of transitional flows for nonlinear fluids \cite{Kalur2022}. The state space typically is an Euclidean vector space.

In this paper, we rely on data to compute $\theta$. We assume that we can sample the solution of $\Pi(x)$ for any (finite) set of initial conditions $x \in \mathbb R^n$ of our choice and observe a finite-horizon response $\{x_k\}_{k=0}^p$ with $p \in \mathbb N$. This approach raises a number of challenges: To begin with, we cannot decide beyond doubt whether the solution is going to converge to the origin, that is, whether $x \in \mathcal R_\infty$. Instead, we are going to introduce a family of finite-time decidable sets $\mathcal R_p \subset \mathcal R_\infty$ approaching $\mathcal R_p$ as $p \to \infty$
and compute an estimate for $\mathcal R_p$ for some given $p$.
Reliance on finite data samples also means that we cannot guarantee that our estimate $R_\theta$ is a true subset of $\mathcal R_p$. We will hence resort to a notion of probabilistic set approximation that allows us to bound the probability that a point is incorrectly predicted to be stable. At last, we need to carefully choose a parametrization that allows us to approximate $\mathcal R_p$ arbitrarily well. To that extent, we will associate $R_\theta$ as the sublevel set of a polynomial function $v$ and $\theta$ with its vector of coefficients.

\subsection*{Objectives}
Let $p \in \mathbb N$ be given.
The goal of this paper is to obtain an estimate $R_\theta$ of $\mathcal R_p$, parametrized by the vector of polynomial coefficients $\theta$,
based on sample trajectories of \eqref{eq:system}
subject to the following probabilistic objectives.
\begin{enumerate}
	\item {\bf Accuracy}: The probability that a point $x \in R_\theta$ is {\em not} in $\mathcal R_p$ is bounded by some $\epsilon > 0$.
	\item {\bf Validity}: The bound $\epsilon$ is given specifically and independently of the polynomial degree.
	\item {\bf Convergence}: With increasing polynomial degree, the estimates $R$ approaches.
\end{enumerate}

Such estimates of the region of attraction can be used as certificates for correct behaviour of a nonlinear system.

\subsection*{Related work}
The problem of estimating $\mathcal R_\infty$ for a given system has a long history in nonlinear analysis. A popular choice are approaches based on polynomial Lyapunov functions. 
Polynomial approximation for $\mathcal R_\infty$ aims to find a sublevel set 
\[
	R_{v,c} = \{ x \in \mathbb R^n \mid v(x) \leq c \}
\] 
where $v: \mathbb R^n \to \mathbb R$ is a polynomial function and $c \geq 0$, that 
\begin{itemize}
	\item minimizes $\varrho_{\mathrm H}(R_{v,c}, \mathcal R_\infty$; and
	\item ensures that $R$ is a subset of $\mathcal R_\infty$.
\end{itemize}
Methods to approximate the region of attraction by constrained polynomial optimisation include \cite{topcuEtAl2008,khodadadiEtAl2014}. A more tractable relaxation is to obtain the shape $v(\cdot)$ by unconstrained optimization, then search for the largest level set that is contained by $R_\infty$. This is the strategy employed in \cite{jonesEtAl2017}. The works cited here rely on dissipativity conditions to provide sufficient conditions for a Lyapunov function and are computationally heavy.

\section{Mathematical Background}
A candidate Lyapunov function is the partial sum
\begin{align}
	V_p(x) =_\text{def} \sum_{k=0}^p \| x_k \|_2^2
\end{align}
where $\{x_k\}_{k \geq 0} \subset \mathbb R^n$ solves $\Pi(x)$
and $p \in \mathbb N \cup \{ \infty \}$. It is easy to see that, if converging, $V_\infty$ is a Lyapunov function and thus, its domain is a subset of $\mathcal R_\infty$. Reverse statements, that is, conditions under which the domain of $V_\infty$ {\em equals} the (true) region of attraction of \eqref{eq:system}, are known as converse theorems \cite{Ortega1973,Jiang2002,Zeng2009,Geiselhart2014}.  In general, the region of attraction is undecidable \cite{blum1998} and hence, there exists no computational form of $V_\infty$. 
%In the work of \cite{Colbert2018}, the asymptotic value of $V_\infty$ has been approximated by finite-horizon simulations with initial conditions and a polynomial Lyapunov function 

\subsection{Truncated Lyapunov function}
We rely on the results of \cite{Balint2006} for a lower bound on the truncated series.
Choose $\tilde p \in \mathbb N$ and $\tilde r > 0$ such that any trajectory $\{x_k\}_{k \geq 0}$ with $x_0 \in \mathscr B^*(\tilde r)$ satisfies
\begin{align*}
	\| x_p \|_2 < \| x_0 \|_2
\end{align*}
for all $p \in \{ \tilde p, \tilde p + 1, \ldots, 2 \tilde p - 1 \}$. 
It can then be shown\footnote{Compare \cite[Lemma~3.3]{Balint2006}.} that any trajectory $\{x_k\}_{k \geq 0}$ with $x_0 \in \mathscr B^*(\tilde r)$ satisfies $x_k \in \mathscr B^*(\tilde r)$ for all $k \geq \tilde p$. %Furthermore, $V_p$ decreases strictly on $\mathscr B^*(\tilde r)$. 
In order to efficiently decide $\mathcal R_p$, we compute a lower bound for $\bar r$ in the appendix.
For any $p \geq \tilde p$, we define %$\mathcal R_p$ as the largest set $R$ such that $V_p( x_0 ) < c_p$ if $x_0 \in R$ with probability one,  
\begin{align}
	\mathcal R_p = \{ x \in \mathbb R^n \mid V_p(x) < c_p \}
\end{align}
where $c_p = (p + 1) \tilde r^2$.
We obtain the following result.

\begin{theorem}
	Let $x \in \mathbb R^n$ be an initial condition for \eqref{eq:system}; 
	$x \in \mathcal R_\infty$ if and only if there exists $p \geq \tilde p$ such that $x \in \mathcal R_p$.
\end{theorem}
\begin{proof}
	See \cite[Theorems~3.7 and 3.10]{Balint2006}.
\end{proof}

In other words, when sampling the solution of $\Pi(x)$ for some given $x \in \mathbb R^n$, we can decide whether or not $x \in \mathcal R_p$ after a finite number of steps. Given $p \geq \tilde p$, the polynomial approximation problem of $\mathcal R_p$ can be written as 
\begin{align}
	\label{eq:problem}
	\min_{v(\cdot)} \sup_{x \in \mathcal R_p} \| v(x) - V_p(x) \|_2
	\quad \text{s.t. $R \subseteq \mathcal R_p$}
\end{align}
where $v$ is a polynomial function with real-valued coefficients and fixed degree.

\subsection{Scenario Optimization}
\label{sec:scenopt}
When using data to solve optimization problems such as \eqref{eq:problem}, deterministic constraints are replaced by so-called chance constraints leading to a probabilistic optimization. The scenario approach of \cite{Tempo2013} then links the sample complexity of the randomized algorithm to its reliability and accuracy with respect to the original, deterministic optimization problem.

\begin{definition}
Let $z \in \rdvar {\mathcal S}$ be a random variable, $\Theta \in \mathbb R^{n_\Theta}$ be a set of parameter vectors, and $f: \mathcal S \to \mathbb R$ be a measurable cost function;
the {\em probabilistic supremum} $\sup^\epsilon f(z)$ for $\epsilon > 0$ is the smallest number $\gamma$ satisfying $\prob{f(z) \leq \gamma} \geq 1 - \epsilon$.
Moreover, the {\em empirical maximum} is given as
\begin{align}
	\max^N f(z) = \max_{i \in \{1, \ldots, N\}} f(z^{(i)})
\end{align}
where $( z^{(1)}, \ldots, z^{(N)} ) \subset \rdvar {\mathcal S}$ is a tuple of $N \in \mathbb N$ independent random variables (the samples) with distribution identical to that of $z$. 
\end{definition}

It can be shown \cite[Theorem~7.4]{Tempo2013} that $\max^N f(z)$ converges almost surely to $\sup f(z)$ if $N \to \infty$, provided that $f$ is continuous at the optimal solution $\hat z$ of $\sup f(z)$ and $\prob {z \in B} > 0$ for any neighbourhood $B \subset \mathcal S$ of $\hat z$. 

When estimating the region of attraction using trajectory samples, we are going to optimize over the empirical maximum of a parametrized cost function. To study accuracy and convergence of our result, we define the probabilistic optimization problems
\begin{subequations}
	\label{eq:probopt}
\begin{alignat}{2}
\label{eq:probopt-deterministic}
&\hat \gamma & {} &= \min_\beta \sup_z g(\beta, z) \\
\label{eq:probopt-probabilistic}
&\gamma_\epsilon & {} &= \min_\beta \sup_z^\epsilon g(\beta, z) \\
\label{eq:probopt-empirical}
&\gamma^N & {} &= \min_\beta \max_z^N g(\beta, z)
\end{alignat}
\end{subequations}
where $\beta \in \Theta$ and $g: \mathcal S \times \Theta \to \mathbb R$ is measurable. 
If either $\Theta$ is compact or $g$ is bounded from below, then $\gamma^N$ is a well-defined random variable with support in $(-\infty, \hat \gamma]$.
Here, \eqref{eq:probopt-empirical} represents the optimization problem our randomized algorithm will solve; Eq.~\eqref{eq:probopt-deterministic} corresponds to the original problem; and \eqref{eq:probopt-probabilistic} is used to formulate the following, probabilistic results.
Denote the optimal solution(s) to \eqref{eq:probopt} by $\arg \hat \gamma$, $\arg \gamma_\epsilon$, and  $\arg \gamma^N$, respectively.

\begin{lemma}
	\label{lem:probopt-convergence}
	Let $\epsilon, \delta > 0$ be confidence levels; then there exists $N \in \mathbb N$ such that 
	\begin{align}
		\label{eq:probopt-convergence}
		\prob {\gamma^N < \gamma_\epsilon} < \delta
	\end{align}
	In addition, let $\Theta$ be compact and convex, $g(\beta, \cdot)$ be convex in $\beta$, and $\beta_N \in \arg \gamma^N$ be unique; if 
	\begin{align}
		\label{eq:probopt-bound}
		\epsilon \, N \geq \frac{e}{e-1} (\log \delta^{-1} + n_\Theta)
	\end{align}
	where $\log(\cdot)$ denotes the natural logarithm and $e$ its base,
	then $\prob {\gamma^N < \sup_z^\epsilon g(\beta_N, z)} < \delta$.
\end{lemma}
\begin{proof}
	See \cite[Theorem~8.1 and Corollary~12.1]{Tempo2013} as well as remarks.
\end{proof}

In other words, the scenario optimization \eqref{eq:probopt-empirical} is $\delta$-reliable in the sense that with probability $1-\delta$ or greater, its optimal solution is $\epsilon$-accurate, that is, $\gamma^N$ is an upper bound of $g(\beta_N, z)$ for $z \in \rdvar {\mathcal S}$ with probability of at least~$1-\epsilon$.
The number of samples $N$ that is necessary to satisfy the confidence levels $\epsilon$ and $\delta$ is called {\em sample complexity}.
We will apply scenario optimization to polynomial regression. Here, $\Theta$ will denote the space of coefficient vectors (up to given degree) and the cost function is linear in the coefficients.

\subsubsection*{Optimization in multiple variables}
We can extend \eqref{eq:probopt} to optimization problems with multiple random variables,\footnote{While we make use of pairs of random variables only, the remainder of this subsection can easily be extended to arbitrary tuples.} viz.
\begin{align}
	\label{eq:probopt-multiple}
	\gamma' = \min_\beta \sup_{z_1, z_2} \max \{ g_1(\beta, z_1), g_2(\beta, z_2) \}
\end{align}
where $z_1 \in \rdvar {\mathcal S_1}$ and $z_2 \in \rdvar {\mathcal S_2}$ are random variables and $g_{1,2}: \Theta \times \mathcal S_{1,2} \to \mathbb R$ are measurable functions that are convex in the first variable. It is easy to see that \eqref{eq:probopt-multiple} is equivalent to the optimization problem in \eqref{eq:probopt-deterministic} for the extended cost function
\begin{align*}
	\bar g: (\beta, z) \mapsto \max \{ g_1(\beta, z_1), g_2(\beta, z_2) \}
\end{align*}
with $z = (z_1, z_2)$, which is again convex in $\beta$. Let again $\gamma_\epsilon$ and $\gamma^N$ denote the optimal values of the probabilistic optimization problems \eqref{eq:probopt-probabilistic} and \eqref{eq:probopt-empirical}, respectively, for $\bar g$. We obtain the following strengthening of Lemma~\ref{lem:probopt-convergence}.

\begin{proposition}
	\label{prop:probopt-multiple}
	Let $\epsilon, \delta > 0$ be confidence levels, let $\Omega$ be compact and convex, and $\beta_N \in \arg \gamma^N$ be unique; then
	\begin{align*}
		\prob {\gamma^N < \sup_{z_i}^\epsilon g_i(\beta_N, z_i)} < \delta
	\end{align*}
	for $i \in \{1,2\}$ if \eqref{eq:probopt-bound} is satisfied by $\epsilon$ and $\delta$.
\end{proposition}
\begin{proof}
	For any $\gamma \in \mathbb R$ and $\beta \in \Theta$, we have that $\bar g(\beta, z) \leq \gamma$ if and only if $g_1(\beta, z_1) \leq \gamma$ and $g_2(\beta, z_2) \leq \gamma$; that is,
	\begin{align*}
		\prob {\bar g(\beta, z) \leq \gamma} &\leq \prob {g_i(\beta, z_i) \leq \gamma}
	\end{align*}
	and hence, $\sup_{z_i}^\epsilon g_i(\beta, z_i) \leq \sup_z^\epsilon \bar g(\beta, z)$ for $i \in \{1,2\}$. The desired result then follows from Lemma~\ref{lem:probopt-convergence}.
\end{proof}

\subsection{Polynomial Set Approximations}
\label{sec:setapprox}
We want to establish polynomial estimates of $\mathcal R_p$ by scenario optimization. As data-driven approaches rarely yield estimates that are guaranteed to be inner or outer approximations, we introduce a probabilistic equivalent. We consider a subset $\mathcal X \subset \mathbb R^n$ of the state space.

\begin{definition}
	Take $\epsilon > 0$ and let $x \in \rdvar {\mathcal X}$ be a random variable; a set $X \subset \mathbb R^n$ is an {\em $\epsilon$-accurate outer approximation} of $\mathcal X$ if and only if $\prob {x \in X} \geq 1 - \epsilon$.
\end{definition}

Intuitively, the probability that $x \in \rdvar {\mathcal X}$ is {\em not} in $X$ is less than $\epsilon$. We define a probabilistic inner approximation likewise.

\begin{definition}
	Take $\epsilon > 0$ and let $x \in \rdvar {(\mathbb R^n \setminus \mathcal X)}$ be a random variable; a set $X \subset \mathbb R^n$ is an {\em $\epsilon$-accurate inner approximation} of $\mathcal X$ if and only if $\prob {x \not\in X} \geq 1 - \epsilon$.
\end{definition}

\begin{figure}[t]
    \center
    \includegraphics[width=.9\linewidth]{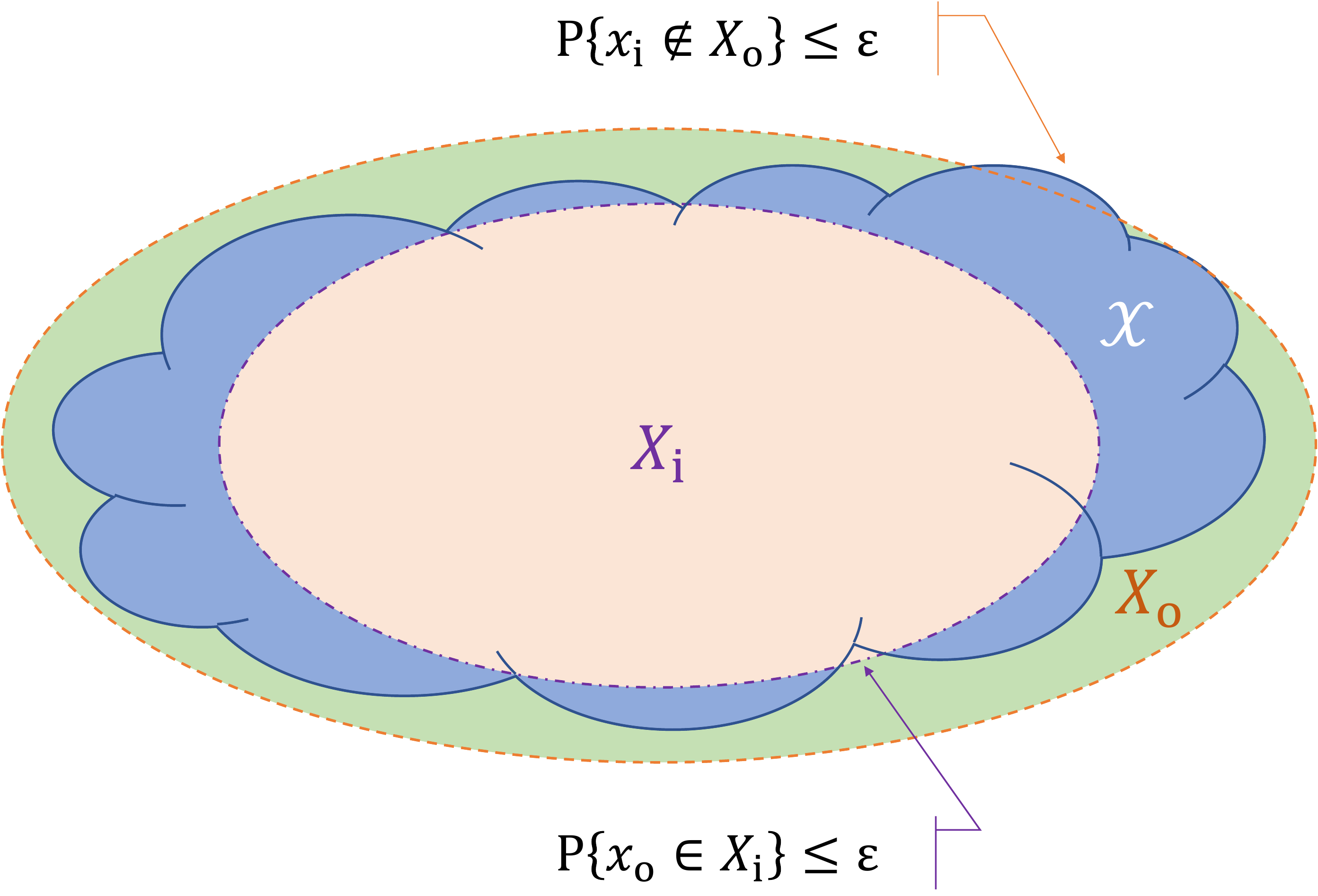}
    \caption{Illustration of $\epsilon$-accurate inner and outer approximations $X_{\mathrm i}$ and $X_{\mathrm o}$ of some set $\mathcal X$ with random variables $x_{\mathrm i} \in \mathcal X$ and $x_{\mathrm o} \in \mathbb R^n \setminus \mathcal X$.}
    \label{fig:approximates}
\end{figure}

The concept of probabilistic inner and outer approximates is illustrated in Fig.~\ref{fig:approximates}. 
Indeed, any sufficiently small (or large) set is likely to be an inner (outer) approximation.\footnote{This is true even in the deterministic case.} Therefore, we aim to find the probabilistic inner approximation which is, at the same time, as close as possible to an accurate outer approximation.
As demonstrated in the following result, the notion of set approximations is tied to that of scenario optimization. 

\begin{proposition}
	\label{prop:setapprox}
	Let $z \in \rdvar {\mathcal X}$ be a random variable, $f: \mathbb R^n \to \mathbb R$ be a measurable function, and $\epsilon, \delta > 0$ be confidence levels; then there exists $N \in \mathbb N$ such that
	\begin{align*}
		X_N = \{ x \in \mathbb R^n \, | \, f(x) \leq \max^N f(z) \}
	\end{align*}
%	where $\gamma^N = \max^N f(z)$,
	is an $\epsilon$-accurate outer approximation of $\mathcal X$ with probability $1 - \delta$ or larger.
\end{proposition}
\begin{proof}
	Define $\gamma^N = \max^N f(z)$;
	by virtue of Lemma~\ref{lem:probopt-convergence}, there exists $N \in \mathbb N$ such that $\prob {\gamma_\epsilon \leq \gamma^N} \geq 1 - \delta$ and $\prob {f(z) \leq \gamma_\epsilon} \geq 1 - \epsilon$,
	where $\gamma_\epsilon = \sup^\epsilon f(z)$ is the probabilistic supremum of $f$ on $z \in \mathcal X$. Hence,
	\begin{align*}
		\prob {\prob {f(z) \leq \gamma^N} \geq 1 - \epsilon} \geq 1 - \delta
	\end{align*}
	which is the desired result by definition of $X_N$.
\end{proof}

It is not hard to derive a similar result for the inner approximation of $\mathcal X$; moreover, an explicit bound for $N$ based on $\epsilon$ and $\delta$ is obtained from \eqref{eq:probopt-bound} with $n_\beta = 0$.
In Proposition~\ref{prop:setapprox}, a probabilistic outer approximation of $\mathcal X$ is given as sublevel set $X_N$ of a function $f$ of which the level is determined by sampling $f$ on $\mathcal X$. The probability of a point $x \in \mathcal X$ not being in the sublevel set thus is equal to the probability of the empirical maximum being smaller than $f(x)$. Hence, accuracy of $X_N$ is linked to scenario optimization.

Clearly, an arbitrary choice for the function $f(\cdot)$ is likely not to lead to a good approximation. The work of \cite{Devonport2020} searched for the smallest ellipsoidal set that is an accurate outer approximation of $\mathcal X$ without any statement for inner approximations. Unfortunately, this approach cannot be easily applied to inner approximations as the empirical maximum of the volume is unbounded.
Instead, we are going to use polynomial functions and compute the vector of coefficients as randomized algorithm.
For the analysis, we apply to the result of Stone and Weierstrass. Here, let $f: \mathcal D \to \mathbb R$ be a given scalar function on $\mathcal D \subset \mathbb R^n$.

\begin{lemma}
	\label{lem:weierstrass}
	Let $\eta > 0$; if $f$ is Lipschitz continuous and $\mathcal D$ is compact, then there exists $q \in \mathbb N$ such that
	\begin{align*}
		\sup_{y \in \mathcal D} | \langle \theta, Z_q\{y\} \rangle - f(y) | \leq \eta
	\end{align*}
	for some vector $\theta \in \mathbb R^{n_q}$.
\end{lemma}
\begin{proof}
	By \cite[Chapter 20, Theorem~3]{Cheney2000}, the polynomials in $y$ form a dense subset of the continuous functions on $\mathcal D$, that is, there exists a sequence of polynomials $\{p: \mathcal D \to \mathbb R\}$ that converges uniformly to $f$. Take $q$ as the degree of $\hat p$ such that $\sup_{y \in \mathcal D} | \hat p(y) - f(y) | \leq \eta$ to complete the proof.
\end{proof}

\section{Methodology}
\label{sec:methodology}
We will present a relaxed solution to the polynomial approximation problem of which the accuracy of the inner approximation is independent of either the polynomial degree or the number of system states.
Our approach is based on obtaining a probabilistic polynomial approximation for $V_p$ on $\mathcal R_p$, while guaranteeing that the resulting estimate is a probabilistic inner approximation. For that purpose, we solve two empirical optimization problems as summarized in Alg.~\ref{alg:main}; the first problem is to obtain a polynomial approximation of $V_p$ and the second problem is to correct for initial conditions outside of $\mathcal R_p$ lying inside the polynomial sublevel set of level $c_p$.

As we are going to show, the result of the probabilistic algorithm is accurate and reliable in the sense that the region of attraction estimate is an $\epsilon$-accurate inner approximation of $\mathcal R_p$ with probability $1-\delta$ or larger, where the confidence levels $(\epsilon, \delta)$ depend on the number of samples but not the degree of the polynomial approximation.
In addition, in the next section, we prove that the estimates converge in probability to $\mathcal R_p$ with increasing polynomial degree and increasing number of samples. 

\begin{algorithm}
	\begin{algorithmic}[1]
		\REQUIRE parameters $p \geq \tilde p$, $q \in \mathbb N$, and $N = (N_1, N_2) \subset \mathbb N$ as well as compact set $\mathcal D \subset \mathbb R^n$
		\STATE evaluate $V_p$ on $N_1$ iid. samples of $\mathcal R_p \times (\mathcal D \setminus \mathcal R_p)$
		\STATE empirically approximate $V_p$ by degree-$2q$ polynomial 
		\STATE evaluate prediction on $N_2$ iid. samples of $\mathcal D \setminus \mathcal R_p$
		\STATE adjust set level to correct for polynomial prediction 
		\ENSURE empirical region of attraction estimate $R_{N,q}$
	\end{algorithmic}
	\caption{Probabilistic region-of-attraction estimation algorithm.}
	\label{alg:main}
\end{algorithm}

\subsection{Randomized Algorithm}
\label{sec:probest}
We define a hierarchy of empirical optimization problems. For any $p \geq \tilde p$, the set $\mathcal R_p$ is bounded by definition of $V_p$. Take a compact set $\mathcal D \subset \mathbb R^n$ with $\mathcal R_p \subsetneq \mathcal D$ and let $\breve x \in \rdvar {\mathcal R_p}$ and $\hat x, \bar x \in \rdvar {(\mathcal D \setminus \mathcal R_p)}$ be random variables; observe that, by definition, $V_p( \bar x ) \geq c_p$ holds surely. In the following problem, we make use of tuples of lengths $N_1, N_2 \in \mathbb N$ comprised of random finite trajectories subject to \eqref{eq:system} with initial conditions $( \breve x_0^{(i)} )_{i=1}^{N_1}$, $( \hat x_0^{(j)} )_{j=1}^{N_1}$, and $( \bar x_0^{(j)} )_{j=1}^{N_2}$ being, respectively, independent and with distribution\footnote{Note that no further assumptions on the nature of this distributions is needed.} identical to $\breve x$, $\hat x$, and $\bar x$.

\begin{problem}
	\label{prob:probest}
	Take $q \in \mathbb N$ and $N = (N_1, N_2)$; solve
	\begin{subequations}
	\begin{align}
		\label{eq:probest-shape}
		\begin{aligned}
		\eta^N = \min_\Theta \max \Big\{ & \max_{\breve x}^{N_1} | \langle \Theta, Z^2_q\{ \breve x \} \rangle - V_p( \breve x ) |, \\
			& \max_{\hat x}^{N_1} \, ( c_p - \langle \Theta, Z^2_q\{ \hat x \} \rangle ) \Big\}
		\end{aligned}
	\end{align}
	where $\Theta \in \mathbb S^+_{n_q}$ is positive semidefinite, and
	\begin{align}
		\label{eq:probest-outer}
		c^N &= \min_{\bar x}^{N_2} \langle \Theta_N, Z^2_q\{\bar  x\} \rangle
	\end{align}
	\end{subequations}
	with $\Theta_N \in \arg \eta^N$.
\end{problem}

\begin{algorithm}[t]
	\input{inner}

	\caption{Compute instance of $\Theta_N \in \arg \eta^N$.}
	\label{alg:probest-inner}
\end{algorithm}

Problem~\ref{prob:probest} can be solved as two convex programs, as detailed in Algorithms~\ref{alg:probest-inner} and \ref{alg:probest-outer}.
Upon solution, we define the empirical region of attraction estimate $R_{N,q}$ as 
\begin{align*}
	R_{N,q} = \{ x \in \mathbb R^n \, | \, \langle \Theta_N, Z^2_q\{ x \} \rangle < c^N \}.
\end{align*}
where $\Theta_N$ and $c^N$ are random variables.
For the remainder of the paper, let $R_{N,q}$ be the result of Problem~\ref{prob:probest}.

\subsection{Accuracy \& Reliability}
Since $c^N$ is based on given and fixed parameters $\Theta_N$, we directly obtain a probabilistic bound on the accuracy of the inner approximation.

\begin{lemma}
	\label{lem:probest-outer-accurate}
	Let $\epsilon, \delta > 0$; if $N_2 \in \mathbb N$ satisfies
	\begin{align*}
		\epsilon N_2 \geq \frac{e}{e-1} \log \delta^{-1}
	\end{align*} 
	then $c^N$ satisfies
	\begin{alignp}
		\label{eq:probest-outer-accurate}
		\prob {\inf^\epsilon \langle \Theta_N, Z^2_q\{ \bar x \} \rangle < c^N} < \delta 
	\end{alignp}
	for any given $\Theta_N$.
\end{lemma}
\begin{proof}
	Since $\mathcal D \setminus \mathcal R_p$ is bounded, $\inf^\epsilon \langle \theta_N, Z_q\{ \bar x \} \rangle \in \mathbb R$.
	Then the desired result follows directly from Lemma~\ref{lem:probopt-convergence}.
\end{proof}

It is important to note that the reliability of $R_{N,q}$ being an $\epsilon$-accurate inner subset of $\mathcal R_p$ depends solely on the number of samples $N_2$ in the second optimization step and is independent on the parameters in the first step.

\begin{remark}
	The purpose of the second argument in \eqref{eq:probest-shape} is to ensure that the polynomial approximation behaves well outside of $\mathcal R_p$, which is used in the following analysis. In practical applications with focus on a probabilistic {\em inner} approximation, the number of samples of $\hat x$ can be reduced without compromising Lemma~\ref{lem:probest-outer-accurate}.
\end{remark}

We state the following corollary for later use.

\begin{corollary}
	\label{cor:probest-outer-accurate}
	Let $\epsilon, \delta > 0$; there exists $N_2 \in \mathbb N$ such that $R_{N,q}$ is an $\epsilon$-accurate inner approximation of $\mathcal R_p$ with probability of at least $1 - \delta$.
	\noqed
\end{corollary}

%For the remainder of the theoretical analysis, we assume that the probability measures associated with $\breve x$ and $\bar x$ satisfy for all $A \subset \mathcal D$,
%\begin{subequations}
%	\label{eq:assumption-sampling}
%\begin{gather}
%	\prob {\breve x \in A} = \frac{\prob {y \in A \cap \mathcal R_p}}{\prob {y \in \mathcal R_p}},
%	\\
%	\prob {\bar x \in A} = \frac{\prob {y \in A - \mathcal R_p}}{\prob {y \in D - \mathcal R_p}},
%\end{gather}
%\end{subequations}
%where $y \in \mathcal D$ is a random variable with $\prob {y \in \mathcal R_p} \in (0,1)$.

\begin{algorithm}[t]
	\input{outer}
	\caption{Compute instance of $c^N$.}
	\label{alg:probest-outer}
\end{algorithm}

\section{Convergence Analysis}
\label{sec:analysis}
We have already established that the region-of-attraction estimate $R_{N,q}$ obtained from solving Problem~\ref{prob:probest} meets the first two of our objectives. We will now demonstrate that $R_{N,q}$ converges, in probability and with respect to a suitable measure of distances between sets, to $\mathcal R_p$. A key step of this proof is to show that $R_{N,q}$ is an outer approximation of a sublevel set of $V_p$ of which the level converges to $c_p$. Here, we need to ensure that the extrapolation of the polynomial approximation beyond $\mathcal R_p$ is well behaving; we thus make the following, technical assumption.

\begin{assumption}
	The random variables $\hat x$ and $\bar x$ are identically distributed.
\end{assumption}

%For any vector $z \in \mathcal D$ with $z \neq 0$, we denote the radial direction of $z$ as $r(z) = \| z \|_2^{-1} z$.
%
%\begin{assumption}
%	$r(\cdot)$ is independent of $\mathcal R_p$, that is, $\prob {r(\breve x) \in B} = \prob {r(\bar x) \in B}$ for all $B \subset \partial \mathscr B^*(1)$.
%\end{assumption}
%
%Both assumptions are satisfied if, for example, the random variable in \eqref{eq:assumption-sampling} is uniformly distributed over $\mathcal D$. 
%With these assumptions,
We are ready to state our main result for the region of attraction estimate $R_{N,q}$.

\begin{theorem}
	\label{thm:probest}
	Let $\epsilon_1, \epsilon_2, \delta_1, \delta_2, \nu > 0$ be given; there exist $N_1, N_2, \allowbreak q \in \mathbb N$ and a set $\mathcal R^\nu \subset \mathcal R_p$ such that
	$R_{N,q}$ is an $\epsilon_2$-accurate inner approximation of $\mathcal R_p$ with probability greater than $1 - \delta_2$ and an $\epsilon_1$-accurate outer approximation of $\mathcal R^\nu$ with probability greater than $1 - \delta_1$, where
	\begin{alignp}
		\label{eq:probest-distance}
		\sup_x \inf_y \| x - y \|_2 < \nu
	\end{alignp}
	with $x \in \mathcal R_p$ and $y \in \mathcal R^\nu$.
\end{theorem}

If, as it remains to be proven, the optimal solutions to Problem~\ref{prob:probest} exist and are unique, it is clear from Lemmas~\ref{lem:probopt-convergence} and \ref{lem:weierstrass} that individually, accuracy of the inner and outer approximations and the goodness of the polynomial fit converge (in probability) if either $N_1 \to \infty$, $N_2 \to \infty$, or $q \to \infty$, respectively. In the remainder of the section we show that convergence (in probability) also holds simultaneously.
To that extent, we will prove:
\begin{enumerate}
	\item That the optimal solution $\Theta_N \in \arg \eta^N$ exists and is unique; that is, we can obtain a bound on the number of samples $N_1$ for given confidence levels $\epsilon$ and $\delta$.
	
	\item That, with increasing polynomial degree $q$, the difference between $c^N$ and $c_p$ converges (in probability) to zero; furthermore, that convergence of the level $\nu$ implies convergence of $\mathcal R^\nu$ in the Hausdorff metric.%\footnote{The Hausdorff distance of two compact sets $A, B \subset \mathbb R^n$ is defined as \[ \max\{ \sup_{x \in A} \operatorname{dist}(x; B), \sup_{y \in B} \operatorname{dist}(y; A) \} \] where $\operatorname{dist}$ denotes the distance between a point and a set.}
	
	\item Combining all of the results, that $R_{N,q}$ converges (in probability) to $\mathcal R_p$ with increasing numbers of samples and polynomial degree.
\end{enumerate}

We start by proving that $R_{N,q}$ is a probabilistic outer approximation of a sublevel set $V_p^{-1}(c_p - \mu) \subset \mathcal R_p$ for some $\mu > 0$. 

\subsection{Existence of the Polynomial Approximation}
Prior to showing uniqueness, we prove that the optimal value $\hat \eta$ of the deterministic $\min \sup$-problem becomes arbitrarily small, provided that the polynomial $q$ is chosen sufficiently large. 
For the first result, we appeal to Lemma~\ref{lem:weierstrass} as well as the structure of $V_p$;
namely, recall that $V_p$ maps $x_0$ to the finite sum of squares of Lipschitz continuous functions 
\begin{align*}
	f_k: x_0 \mapsto x_k 
\end{align*}
where $k \in \{ 0, \ldots, p \}$ and $\{x_k\}_{k \geq 0}$ is the solution of $\Pi(x_0)$. 

\begin{proposition}
	\label{prop:weierstrass-sos}
	Let $\eta > 0$; there exists $q \in \mathbb N$ such that
	\begin{align}
		\label{eq:probest-weierstrass}
		\begin{aligned}
			\sup_{\breve x, \hat x} \max \left\{
			| \langle \Theta, Z^2_q\{\breve x\} \rangle - V_p(\breve x) |, \right. \quad& \\
			\left. (c_p - \langle \Theta, Z^2_q\{\hat x\} \rangle)
			\right\} &\leq \eta
		\end{aligned}
	\end{align}
	for some matrix $\Theta \in \mathbb S_{n_q}^+$.
\end{proposition}
\begin{proof}
	Choose $\mu > 0$ with $\mu(\mu + 2 F) < \eta/p$, where $F$ is an upper bound of $\| f_1 \|_2, \ldots, \| f_p \|_2$ on $\mathcal D$.
	By Lemma~\ref{lem:weierstrass}, there exist $q \in \mathbb N$ and $\theta_1, \ldots, \theta_p \in \mathbb R^{n_q \times n}$ such that $\sup_{x \in \mathcal D} | \theta_k^{\mathrm T} Z_q\{x\} - f_k(x) | \leq \mu$ for all $k \in \{ 1, \ldots, p \}$. Then
	\begin{multline*}
		\Big| \sum_{k=1}^p \| \theta_k^{\mathrm T} Z_q\{x\} \|_2^2 - \| f_k(x) \|_2^2 \Big| \\ %\hspace{11em} \\
		\leq %\sum_{k=1}^p \Big| \| \theta_k^{\mathrm T} Z_q\{\breve x\} \|_2^2 - \| f_k(\breve x) \|_2^2 \Big| \\ % triangle inequality
%		= 
		\sum_{k=1}^p \Big| ( \theta_k^{\mathrm T} Z_q\{x\} - f_k(x) )^{\mathrm T} ( \theta_k^{\mathrm T} Z_q\{x\} + f_k(x) ) \Big| \\
		\leq \sum_{k=1}^p \| \theta_k^{\mathrm T} Z_q\{x\} - f_k(\breve x) \|_2 \cdot \| \theta_k^{\mathrm T} Z_q\{x\} + f_k(x) \|_2 \\ % Cauchy-Schwarz inequality
		\leq \mu \sum_{k=1}^p (2 \| f_k(x) \|_2 + \mu) \leq \eta
	\end{multline*}
	holds for all $x \in \mathcal D$. 
	Note that $V_p(x') \geq c_p$ if $x' \in \mathcal D \setminus \mathcal R_p$ by definition of $\mathcal R_p$.
	Let $\theta_0$ be the vector that gives $\theta_0^{\mathrm T} Z_q\{ x \} \equiv x$ and define $L_k = \theta_k^{} \theta_k^{\mathrm T}$. Then $L_0, L_1, \ldots, L_p \in \mathbb S_{n_q}^+$ and $\Theta = L_0 + \sum_{k=1}^p L_k$ is the desired result.
\end{proof}

We note an immediate result of Proposition~\ref{prop:weierstrass-sos} for the empirical optimization problem in \eqref{eq:probest-shape}.

\begin{lemma}
	\label{lem:probest-inner-exists}
	Let $\eta > 0$; there exists $q \in \mathbb N$ such that $\eta^N \leq \eta$ holds surely.
\end{lemma}
\begin{proof}
	Denote the cost function in \eqref{eq:probest-weierstrass} by $g(\Theta, x)$ with $x = (\breve x, \hat x)$.
	By Proposition~\ref{prop:weierstrass-sos}, there exist $q \in \mathbb N$ and $\hat \Theta \in \mathbb S^+_{n_q}$ 
%	such that $\hat \eta = \sup_{\breve x} | \langle \hat \Theta, Z^2_q\{ \breve x \} - V_p(\breve x) | \leq \eta$ and hence, $\eta^N \leq \max_{\breve x}^{N_1} | \hat \Theta, Z^2_q\{ \breve x \} - V_p(\breve x) | \leq \hat \eta$ 
	such that $\hat \eta = \sup_x g(\hat \Theta, x) \leq \eta$ and hence, $\eta^N \leq \max_x^{N_1} g(\hat \Theta, x) \leq \hat \eta$ 
	holds surely. 
\end{proof}

\subsection{Uniqueness of the Polynomial Approximation}
In order to apply Lemma~\ref{lem:probopt-convergence}, we need that the optimal solution $\Theta_N \in \arg \eta^N$ of the optimization problem in \eqref{eq:probest-shape} is unique. Here, we will assume without loss of generality that the polynomial regression is overdetermined. 
For the following result, note that $(\eta^N, \Theta_N)$ is the solution of a convex program (Alg.~\ref{alg:probest-inner}) of which the feasible set is the intersection of $K_1 = \mathbb R \times \mathbb S^+_{n_q}$ and $K_2 = \{ \xi \in \mathbb R^{n_{2q} + 1} \, | \, A \xi \geq b \}$, where the rows of $A \in \mathbb R^{3N_1 \times (n_{2q}+1)}$ and $b \in \mathbb R^{3N_1}$ are given as
\begin{subequations}
	\label{eq:probest-inner-constraints}
\begin{alignat}{2}
	\label{eq:probest-inner-constraints-lower}
	A_i^{\mathrm T} &= \begin{bmatrix} 1 & +Z_{2q}\{\breve x^{(i)}\} \end{bmatrix}, \quad & b_i &= +v^{(i)} \\
	\label{eq:probest-inner-constraints-upper}
	A_{N+i}^{\mathrm T} &= \begin{bmatrix} 1 & -Z_{2q}\{\breve x^{(i)}\} \end{bmatrix}, \quad & b_{N+i} &= -v^{(i)} \\
	\label{eq:probest-inner-constraints-exterior}
	A_{2N+j}^{\mathrm T} &= \begin{bmatrix} 1 & +Z_{2q}\{\hat x^{(j)}\} \end{bmatrix}, \quad & b_{2N+j} &= c_p
\end{alignat}
\end{subequations}
for all $i, j \in \{1, \ldots, N_1\}$. %if $\eta > 0$, and hold with equality otherwise. 
The objective can be written as $\eta = \langle c, \xi \rangle$ for a suitable linear form $c \in \mathbb R^{n_{2q}+1}$. 

\begin{remark}
	Some monomials of $x$ appear repeatedly in the square matrix $Z_q^2\{x\}$, thus adding some ambiguity to the solution $\Theta \in \mathbb S_{n_q}^+$. Since this is not reflected in the polynomial $\langle \Theta, Z_q^2\{x\}$, in \eqref{eq:probest-inner-constraints} we have (with some abuse of notation) used the vector $Z_{2q}\{x\}$ of same degree instead. 
\end{remark}

For some optimal solution $\bar \xi \in \arg \min_{\xi \in K_1 \cap K_2} \langle c, \xi \rangle$, we denote by $\bar A$ and $\bar b$ those constraints of \eqref{eq:probest-inner-constraints} that $\bar \xi$ satisfies with equality. Observe that any set containing $c$ and up to $n_{2q}$ linearly independent rows of $A$ is again linearly independent and, by \eqref{eq:probest-shape}, $\bar A$ includes at least rows with both positive and negative monomials; the corresponding samples $\breve x^{(i)}$ (resp., $\hat x^{(j)}$) with $i, j \in \{1, \ldots, N_1 \}$ are mutually exclusive; as well as $\langle c, \bar \xi \rangle \geq 0$. 

\begin{proposition}
	\label{prop:probest-inner-unique}
	The optimal solution $\Theta_N$ is unique if the associated matrix $\bar A$ has rank $n_{2q}+1$.
\end{proposition}
\begin{proof}
	We assume without loss of generality that $\bar A$ has at most $n_{2q}+1$ rows. 
	Let $\bar \xi \in \arg \min_{\xi \in K_1 \cap K_2} \langle c, \bar \xi \rangle$; that is, $\langle c, \bar \xi \rangle \leq \langle c, \xi \rangle$ for all $\xi \in K_1 \cap K_2$, which is just the definition for $c$ to be element of the {\em inwards} normal cone, $N^-_{K_1 \cap K_2}(\bar \xi)$, of $K_1 \cap K_2$ at $\bar \xi$. We consider first the case that $\bar \xi$ lies in the interior of $K_1$, that is, $\bar \xi$ is the solution to the linear program $\min_{\xi \in K_2} \langle c, \xi \rangle$ and $N^-_{K_1 \cap K_2}(\bar \xi) = N^-_{K_2}(\bar \xi)$. 
	Now, $\bar \xi$ is unique if (and only if) $c + \varepsilon \in -N_{K_2}(\bar \xi)$ for any $\varepsilon \in \mathbb R^{n_{2q}+1}$ that is small enough \cite[Theorem~1]{Mangasarian1979}.
	
	Observe that the cone $N^-_{K_2}(\bar \xi) \subset \mathbb R^{n_{2q}+1}$ is spanned\footnote{A cone $C$ is spanned by the set of vectors $\{ a_1, ..., a_\ell \}$ if and only if all elements $y \in C$ satisfy $y = \lambda_1 a_1 + \cdots + \lambda_\ell a_\ell$ with $\lambda_1, \ldots, \lambda_\ell \geq 0$.} by the rows of $\bar A$.\footnote{See Appendix~\ref{app:conespan} for a proof.} %\footnote{This is a consequence of the cone property and Farkas' lemma.} 
	Since any set containing $c$ and (at most) $n_{2q}$ rows of $\bar A$ is linearly independent, we now claim that $\bar A$ is square; for otherwise, $c \in N^-_{K_2}(\bar \xi)$ could be written as linear combination of $n_{2q}$ rows of $\bar A$, which contradicts the linear independence. With the same argument, we conclude that $c = \bar A^{\mathrm T} \lambda$ with $\lambda \in \mathbb R_{> 0}^{n_{2q}+1}$. As $\bar A$ is square and full rank, its range is equal to $\mathbb R^{n_{2q}+1}$. Hence, if $\| \varepsilon \|_2$ is small enough, there exists $\kappa \in \mathbb R^{n_{2q}+1}$ satisfying $\bar A^{\mathrm T} \kappa = \varepsilon$ and $\lambda + \kappa \in \mathbb R_{\geq 0}^{n_{2q}+1}$; in other words, $c + \varepsilon \in N^-_{K_2}(\bar \xi)$.
	
	To complete the proof, assume that $\bar \xi$ lays on the boundary of $K_1$; but $\langle c, \xi \rangle$ is unbounded on $\xi \in K_1$ and hence, uniqueness of $\bar \xi$ is not affected.
\end{proof}
%\begin{proof}
%	Let $\bar \xi \in \arg \min_{\xi \in K} \langle c, \xi \rangle$; that is, $\langle c, \bar \xi \rangle \leq \langle c, \xi \rangle$ for all $\xi \in K$, which is just the definition for $c$ to be element of the {\em inwards} normal cone, $-N_K(\bar \xi)$, of $K$ at $\bar \xi$. Furthermore, if the matrix $[A_1, \ldots, A_N]^{\mathrm T}$ (resp., $[A_{N+1}, \ldots, A_{2N}]^{\mathrm T}$) has more then $n_{2q}$ linearly independent rows,\footnote{This is the case if $( \breve x^{(i)} )_{i=1}^N$ has at most $N-n_q$ duplicates.} then $\bar \xi$ satisfies (at least) $n_{2q} + 1$ constraints with equality and thus $\bar \xi$ is a vertex of $K$.
%	
%	Suppose now that $\bar \xi_1, \bar \xi_2 \in K$ are optimal solutions. By convexity of $K$ then, $\xi_\lambda := \lambda \bar \xi_1 + (1-\lambda) \bar \xi_2 \in K$ and
%	\begin{align*}
%		\langle c, \xi_\lambda \rangle = \lambda \langle c, \bar \xi_1 \rangle + (1-\lambda) \langle c, \bar \xi_2 \rangle \leq 0
%	\end{align*}
%	for all $\lambda \in (0, 1)$, since $c \in -N_K(\bar \xi_1) \cap -N_K(\bar \xi_2)$, hence $c \in -N_K(\xi_\lambda)$. That is, any $\xi_\lambda$ is optimal and thus a vertex, implying that $\xi_\lambda = \bar \xi_1 = \bar \xi_2$, the desired result.
%\end{proof}

Uniqueness of $\Theta^N$ then is given, except for some events of measure zero, if the number of samples is large enough. 

\begin{corollary}
	\label{cor:probest-inner-unique}
	If $N_1 \geq n_{2q}$, the optimal solution $\Theta_N$ is almost surely unique.
\end{corollary}
\begin{proof}
	Any set of up to $n_{2q}$ monomial vectors $z_1^{(i)}$ (resp., $z_2^{(j)}$) with $i, j \in \{1, \ldots, N_1 \}$ is linearly independent with probability one (see Appendix~\ref{app:vandermonde}); thus, any set of $c$ and up to $n_{2q}$ rows of $A$ is linearly independent and $\bar A$ has at least $n_{2q}+1$ rows.
	Obtain the matrix $\Lambda$ from $\bar A$ by multiplying each row that has $-z_1^{(i)}$ by $-1$ (such that the second column of $\Lambda$ is all $1$) as well as reordering its rows until the top-left $2 \times 2$ block reads
	\begin{align*}
		\Lambda^{[2]} = \begin{bmatrix} 1 & 1 \\ -1 & 1 \end{bmatrix}
	\end{align*}
	and note that $\Lambda$ and $\bar A$ are of equal rank. Then one applies the induction in the proof of Proposition~\ref{prop:vandermonde} (appendix), starting with the nonsingular block $\Lambda^{[2]}$, to verify that $\Lambda$ has almost surely full rank.
	That is, $\bar A$ has rank $n_{2q}+1$ and uniqueness follows by virtue of Proposition~\ref{prop:probest-inner-unique}.
\end{proof}
%\begin{proof}
%	Any tuple $( \breve x^{(i)} )_{i=1}^N$ with $\breve x^{(i)} = \breve x^{(j)}$ for at least one pair $i, j \in I \subset \{1, \ldots, N \}$, where $i \neq j$, forms a hyperplane in the $N$-dimensional space $\mathcal D^N$ and thus has measure zero. As there are finitely many permutations, the union over all nonempty sets $I \subset \{1, \ldots, N\}$ is a null set and $\prob {\breve x^{(i)} = \breve x^{(j)}} = 0$ for all $i \neq j$. Hence, the set $\{ \breve x^{(i)} \}_{i=1}^N$ has more than $n_q$ elements, unless $N < n_q$, with probability one.
%\end{proof}

In the remainder of this section we will now conclude that the empirical polynomial approximation converges (in probability) to $V_p$.

\subsection{Accuracy of the Polynomial Approximation}
Having established existence and uniqueness, we obtain that the polynomial fit $\Theta_N$ is probabilistically accurate. We then proceed to show that the empirical bound $c^N$ does not add unnecessary conservatism.
Furthermore, the sublevel set can be chosen arbitrarily close to $\mathcal R_p$, where we assume that the distance $\nu$ is much smaller than $\sqrt{c_p}$.

\begin{lemma}
	\label{lem:probest-inner-accurate}
	Let $\epsilon, \delta, \eta > 0$; there exist $N_1, q \in \mathbb N$ such that $(\Theta_N, \eta^N)$ satisfy
	\begin{alignp}
		\label{eq:probest-inner-accurate}
		\left.
		\begin{aligned}
%		\prob { | \langle \Theta_N, Z^2_q\{ \breve x \} \rangle - V_p(\breve x) | > \eta^N } < \epsilon
		&\prob { \sup_\epsilon | \langle \Theta_N, Z^2_q\{ \breve x \} \rangle - V_p(\breve x) | \leq \eta^N } \\
		&\prob {\sup_\epsilon (c_p - \langle \Theta_N, Z^2_q\{ \hat x\} \rangle) \leq \eta^N}
		\end{aligned}
		\right\} \geq 1 - \delta
	\end{alignp}
%	with probability greater than $1-\delta$, 
	and $\eta^N \leq \eta$.
\end{lemma}
\begin{proof}
	By Lemma~\ref{lem:probest-inner-exists}, there exists $q \in \mathbb N$ such that $\eta^N \leq \eta$. Choosing $N_1 \geq \max \{ n_{2q}, \epsilon^{-1} (\log \delta^{-1} + n_{2q}) (1-e^{-1})^{-1} \}$, the optimal solution $\Theta_N$ is unique with probability one (Corollary~\ref{cor:probest-inner-unique}) and thus, the desired results follow from Proposition~\ref{prop:probopt-multiple}.
\end{proof}

Since the empirical polynomial approximation is $\epsilon$-accurate (with reliability $\delta > 0$), we obtain a probabilistic bound on the difference between $c^N$ and $c_p$. Here, we make use of the fact that $\hat x$ and $\bar x$ are equally distributed. %$\Theta_N$ is constrained to be a positive definite matrix.
%Let $x_{\mathcal R}^{(j)}$ denote the projection of $\bar x^{(j)}$ onto $\mathcal R_p$; thus, since $\Theta_N$ is positive definite, $\langle \Theta_N, Z^2_q\{ x_{\mathcal R}^{(j)} \} \rangle \leq \tau^{(j)}$ holds surely.

\begin{lemma}
	\label{lem:probest-outer-gap}
	Let $\epsilon > 0$; %there exists $N_1 \in \mathbb N$ such that $\eta^N$ satisfies
	provided that \eqref{eq:probest-inner-accurate} holds, $c^N$ satisfies
	\begin{alignp}
		\label{eq:probest-outer-gap}
		c^N < c_p - \eta^N
	\end{alignp}
	with probability less than $1 - (1 - \epsilon)^{N_2}$.
\end{lemma}
\begin{proof}
	%By Lemma~\ref{lem:probest-inner-accurate}, there exists $N_1 \in \mathbb N$ such that $\eta^N$ satisfies, with probability greater than $1-\delta$,
	Assume \eqref{eq:probest-inner-accurate} holds; $\eta^N$ and $\Theta_N$ then satisfy
	\begin{align*}
		\prob {\langle \Theta_N, Z_q^2\{\bar x^{(j)}\} \rangle - c_p \leq \eta^N, \; j \in \{1, \ldots, N\}} \geq (1-\epsilon)^N
	\end{align*}
	as well. %for any $\epsilon > 0$. 
	As $c^N \leq \langle \Theta_N, Z_q^2\{\bar x^{(j)}\} \rangle$ holds surely for any $j \in \{1, \ldots, N\}$, this completes the proof.
\end{proof}

Before proving our main theorem, we derive a bound on the Hausdorff distance\footnote{The Hausdorff distance of $A, B \subset \mathbb R^n$ simplifies to \[ \sup_{x \in B} \operatorname{dist}(x; A) \] if $A \subseteq B$.} between $\mathcal R^\nu$ and $\mathcal R_p$ based on the sublevel $c_p - \mu > 0$.

\begin{proposition}
	\label{prop:probest-inner-gap}
	Let $\nu > 0$; there exists $\mu > 0$ such that
	\begin{alignp}
		\label{eq:probest-inner-gap}
		\inf_y \| x - y \|_2 < \nu
	\end{alignp}
	with $y \in V_p^{-1}(c_p - \mu)$ for all $x \in \mathcal R_p$.
\end{proposition}
\begin{proof}
	Since $\phi$ is Lipschitz continuous and $\mathcal R_p$ is bounded, there exists $\ell \geq 1$ such that
	\begin{align*}
		\| x \|_2^2 \leq V_p(x) \leq \ell \| x \|_2^2
	\end{align*}
	for all $x \in \mathcal R_p$. Choose $\mu < c_p - \ell^2(c_p - 2\nu \sqrt{c_p} + \nu^2)$ and denote $r_1 = \sup_x \| x \|_2$ and $r_2 = \inf_y \| y \|_2$; then
	$\sqrt {c_p} > r_1 > r_2 \geq \ell^{-1} \sqrt{c_p - \mu}$. Hence, $\inf_y \| x - y \|_2 \leq r_1 - r_2 < \nu$ %\sqrt {c_p} - \ell^{-1} \sqrt{c_p - \mu}$ 
	is satisfied, the desired result.
\end{proof}

\subsection{Convergence of the Polynomial Approximation}
We conclude the theoretical analysis by proving that $R_{N,q}$ converges (in probability) to an inner approximation of $\mathcal R_p$ and an outer approximation of $\mathcal R^\nu$, where the sublevel set $\mathcal R^\nu$ converges (in the Hausdorff metric) to $\mathcal R_p$.

\subsubsection*{Proof of Theorem~\ref{thm:probest}}
%\begin{proof}
	By Lemma~\ref{lem:probest-outer-accurate}, there exists $N_2 \in \mathbb N$ such that \eqref{eq:probest-outer-accurate}, that is, $R_{N,q}$ satisfies
	\begin{alignp}
		\label{eq:probest-outer}
		\prob {\bar x \in R_{N,q}} < \epsilon_2
	\end{alignp}
	with probability greater than $1-\delta_2$. Recalling that $\bar x \not\in \mathcal R_p$, \eqref{eq:probest-outer} is the definition of an $\epsilon_2$-accurate inner approximation of $\mathcal R_p$. Let now $\mathcal R^\nu$ be the largest set $R \subset \mathcal R_p$ such that $y \in R$ if $V_p(y) \leq c_p - \mu$ with probability one, $x \in \mathcal R_p$, and $y \in \mathcal R^\nu$, where $\mu > 0$ is chosen according to Proposition~\ref{prop:probest-inner-gap} such that \eqref{eq:probest-inner-gap} holds, implying that \eqref{eq:probest-distance} is satisfied. 
	
	Choose $\eta \leq \mu/2$; assume that %\eqref{eq:probest-inner-accurate} and \eqref{eq:probest-outer-gap} hold as well as 
	$\eta^N \leq \eta$. If $y \in \mathcal R^\nu$, then
	\begin{alignp}
		\label{eq:probest-inner}
	\begin{aligned}
		\prob {y \not\in R_{N,q}} 
		&\leq \prob {\langle \Theta_N, Z^2_q\{y\} \rangle - V_p(y) > c^N - (c_p - \mu)} \\
		&\leq \prob {\langle \Theta_N, Z^2_q\{y\} \rangle - V_p(y) > \mu - \eta} \\
		&\leq \prob {| \langle \Theta_N, Z^2_q\{y\} \rangle - V_p(y) | > \eta^N} \leq \epsilon_1
	\end{aligned}
	\end{alignp}
	where the first inequality follows from the definitions of $R_{N,q}$ and $\mathcal R^\nu$, the second from \eqref{eq:probest-outer-gap}, and the third from \eqref{eq:probest-inner-accurate}. By Lemma~\ref{lem:probest-inner-accurate} and \ref{lem:probest-outer-gap}, there exist $N_1, q \in \mathbb N$ such that \eqref{eq:probest-inner} holds with probability greater than $(1-\delta') (1 - \epsilon_1)^{N_2}$ for some $\delta' > 0$. Take 
	\[ \delta' \leq 1 - (1 - \delta_1) (1 - \epsilon_1)^{-N_2} \]
	then $R_{N,q}$ is an $\epsilon_1$-accurate inner approximation of $\mathcal R^\nu$ with probability greather than $1 - \delta_1$. 
	
	This concludes the proof of Theorem~\ref{thm:probest}. \hfill \QED
%\end{proof}

\section{Numerical Examples}
\label{sec:examples}
We present examples from optimization-based and neural network control. The closed-loop dynamics here are partially written
\begin{align*}
	x_{k+1} = A_\text{op} x_k + \psi(x_k)
\end{align*}
where $A_\text{op}$ denotes the open-loop linear part and $\psi$ is a nonlinear saturation or activation function. In the examples, $A_\text{op}$ is not necessarily stable but $\psi$ is differentiable around the origin, leading to the closed-loop dynamics \eqref{eq:system} with
\begin{gather*}
	A = A_\text{op} + J_0 \\
	\phi = \psi - J_0 x
\end{gather*}
where $J_0$ is the Jacobian of $\psi$ at the origin.

\subsection{Saturated LQR}
We consider the classical problem of an open-loop unstable linear system
\begin{align}
	\label{eq:example}
	A_\text{op} = \begin{bmatrix} 1.0745 & 0.1025 \\ 1.5079 & 1.0745 \end{bmatrix}, \quad\quad B = \begin{bmatrix} 0.1518 \\ 3.0741 \end{bmatrix}
\end{align}
under saturated LQR control (with $Q = I_2$ and $R = 1$)
\begin{subequations}
	\label{eq:example-lqr}
\begin{align}
	K &= \begin{bmatrix} -0.7999 & -0.3397 \end{bmatrix} \\
	\psi(x) &= B \operatorname{sat}(K x)
\end{align}
\end{subequations}
where $\operatorname{sat}(\cdot)$ is the scalar projection to the interval $[-1, 1]$. The closed-loop linear part $A = A_\text{op} + BK$ is stable and satisfies $|| A ||^p < 1$ for all $p \geq \tilde p = 5$. Furthermore, a lower bound $r_\iota = 1.5052$ with $\iota = 10^{-6}$ has been obtained for $\tilde r$ using the approach in Appendix~\ref{app:invariance}.

\sisetup{retain-unity-mantissa = false, scientific-notation = fixed, fixed-exponent = -2}

We want estimate the set $\mathcal R_{25}$ for \eqref{eq:example-lqr} with $c_{25} = 58.9031$. To that extent, we run Algorithms~\ref{alg:probest-inner} and \ref{alg:probest-outer} using $N_1 \in \{100, 250, 500, 1000 \}$ and $N_2 = 2 N_1$ samples, respectively, and a polynomial degree $2q = 4$. Based on a reliability of $\delta_1 = \delta_2 = \num[scientific-notation=false]{1e-6}$, these sample sizes correspond to accuracies $\epsilon_2$ ranging from \numrange{0.1093}{0.0109} for inner estimation (relative to $\mathcal R_{25}$) and $\epsilon_1$ from \numrange{0.7881}{0.0788} for outer estimation (relative to the inner sublevel set $\mathcal R^\nu$).
Recognizing the randomized nature of our approach, we have computed each estimate $R_{N,q}$ multiple times. Fig.~\ref{fig:scenarioopt-lqr} shows the estimates $R_{N,q}$ for the saturated LQR.

\begin{figure}
	\center
	
%	\setlength\FigWidth{\linewidth}
%	\setlength\FigHeight{.7\linewidth}
%	
%%	\pgfplotsset{label style={/tikz/font=\small}, ticklabel style={/tikz/font=\small}}
%	\pgfplotsset{every axis/.append style={/tikz/font=\small}}
%	
%	\tikzsetnextfilename{lqr/scenarioopt}
	\includegraphics{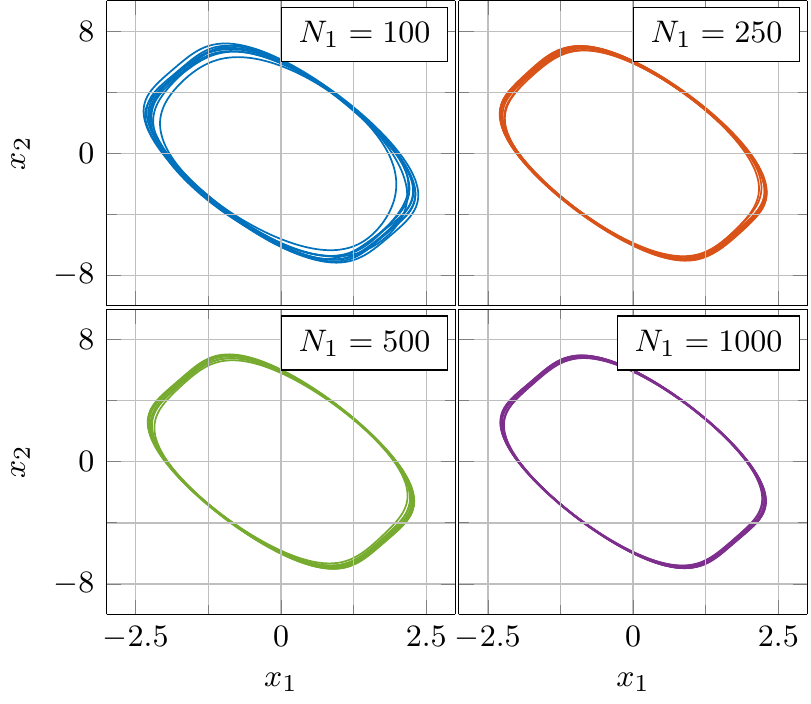}
	
	\caption{Randomized estimates $R_{N,q}$ of the set $\mathcal R_{25}$ for \eqref{eq:example}, \eqref{eq:example-lqr} with varying sample sizes and polynomial degree $2q = 4$.}
	\label{fig:scenarioopt-lqr}
\end{figure}

We evaluate $R_{N,q}$ empirically at $M = \num[scientific-notation=false]{10000}$ sample points $x \in \mathcal D$ for each estimate. We are interested in those samples that lie in $\mathcal R_{25}$ but not in $R_{N,q}$ (false-negative), and those that lie outside $\mathcal R_{25}$ but within $R_{N,q}$ (false-positive). 
For each sample size $N_2 = 2N_1$, the average and worst-case empirical probabilities (denoted by $\prob[M]{\cdot}$) are detailed in Tab.~\ref{tab:scenarioopt-lqr}.
The rate of false-positive test points decreases for larger sample sizes, in agreement with the theoretical results, and is well below the predicted accuracies in all cases. The rate of false-negative test points (relative to $\mathcal R_{25}$) shows more variation for smaller samples than for larger, which is also to be expected.

\begin{table}[h]
	\center
	
	\caption{Empirical accuracy of $R_{N,q}$ for \eqref{eq:example-lqr} with varying sample sizes and degree $2q = 4$ (results in \num{0.01}).}
	\label{tab:scenarioopt-lqr}
	
	\sisetup{table-format=2.2, table-omit-exponent}
	\begin{tabular}{| c | S S[table-space-text-pre={(}] | S S |}
		\hline
		 & \multicolumn{2}{S |}{{$\prob[M]{\mathcal R_p \setminus R_{N,q}$}}} & \multicolumn{2}{S |}{{$\prob[M]{R_{N,q} \setminus \mathcal R_p}$}} \\
		$N_1$ & {mean} & {max} & {mean} & {max} \\
		\hline
		$100$ & 0.0473 & 0.1405 & 0.0040 & 0.0168 \\
		$250$ & 0.0389 & 0.0690 & 0.0019 & 0.0048 \\
		$500$ & 0.0453 & 0.0892 & 0.0009 & 0.0031 \\
		$1000$ & 0.0441 & 0.0671 & 0.0004 & 0.0013 \\
		\hline
	\end{tabular}
\end{table}

\subsection{Suboptimal MPC}
We now turn our attention to a time-distributed optimal control scheme for the open-loop dynamics in \eqref{eq:example}. Let $z \in \mathbb R^3$ be a finite-horizon input sequence. The optimal MPC feedback is obtained by solving an optimal control problem parametrized in the initial condition $x$, namely, minimizing the finite-horizon ($T = 3$) quadratic cost 
\begin{align*}
	\langle P \xi_T, \xi_T \rangle + \sum_{k=0}^{T-1} \langle Q \xi_k, \xi_k \rangle + \langle R \mu_k, \mu_k \rangle
\end{align*}
where $\{\xi_k\}_{k=0}^T$ is a solution to $\xi_{k+1} = A_\text{op} \xi_k + B \mu_k$ for all $k \in \{0, \ldots, T\}$ with $\xi_0 = x$ under control inputs $\{ \mu_k \}_{k=0}^{T-1} \subset [-1, 1]$. The terminal weight $P$ is the solution to the discrete Riccati equation for $Q = I_2$ and $R = 1$.
Setting $z = (\mu_0, \ldots, \mu_{T-1})$ and eliminating the state sequence algebraically, we obtain the quadratic program
\begin{subequations}
	\label{eq:mpc-optimal}
\begin{align}
	\hat z &\in \operatorname*{arg\,min}_{z \in [-1, 1]^3} \langle H z, z \rangle + \langle 2 G x, z \rangle \\
	\hat u &= \begin{bmatrix} 1 & 0 & 0 \end{bmatrix} \, \hat z
\end{align}
\end{subequations}
with suitable matrices $G \in \mathbb R^{T \times 2}$ and $H \in \mathbb S_T$ (see \cite{Leung2021b} for details).

We employ a suboptimal solution to \eqref{eq:mpc-optimal} given by $r = 25$ iterations of the projected-gradient descent algorithm
\begin{subequations}
	\label{eq:mpc-suboptimal}
\begin{align}
	z_{i+1} &= \Pi( z_i - 2\alpha( H z_i + G x ) ) \\
	\psi(x,z) &= \begin{bmatrix} 1 & 0 & 0 \end{bmatrix} \, z_r \quad \text{if $z_0 = z$}
\end{align}
\end{subequations}
where $\Pi$ is the projection onto the input constraints. The closed-loop dynamics of \eqref{eq:example} and \eqref{eq:mpc-suboptimal} form a linear system with combined state $(x, z)$ under an extended nonlinear projection operator. We obtain $r_\iota = 0.6132$ for $\tilde p = 6$.

Choosing $z = 0$ as initial guess, we estimate the set $\mathcal R_{250}$ with $c_{250} = 94.3704$ restricted to the space of $x$. We run Algorithms~\ref{alg:probest-inner} and \ref{alg:probest-outer} using $N_1 = 1500$ and $N_2 = 3000$ samples, respectively, and polynomial degrees $2 q \in \{ 2, 4, 6, 8 \}$. Based on a reliability of $\delta_1 = \delta_2 = \num[scientific-notation=false]{1e-6}$, these sample sizes correspond to an accuracy of $\epsilon_2 = \num{0.0073}$ for inner estimation and accuracies $\epsilon_1 \in [\num{0.0241}, \num{0.2519}]$, depending on $n_q$, for outer estimation (relative to the inner sublevel sets $\mathcal R^\nu$). Fig.~\ref{fig:scenarioopt-mpc} shows the estimates $R_{N,q}$ for the suboptimal MPC.

\begin{figure}
	\center
	
%	\setlength\FigWidth{\linewidth}
%	\setlength\FigHeight{.7\linewidth}
%	
%%	\pgfplotsset{label style={/tikz/font=\small}, ticklabel style={/tikz/font=\small}}
%	\pgfplotsset{every axis/.append style={/tikz/font=\small}}
%	
%	\tikzsetnextfilename{mpc/scenarioopt}
	\includegraphics{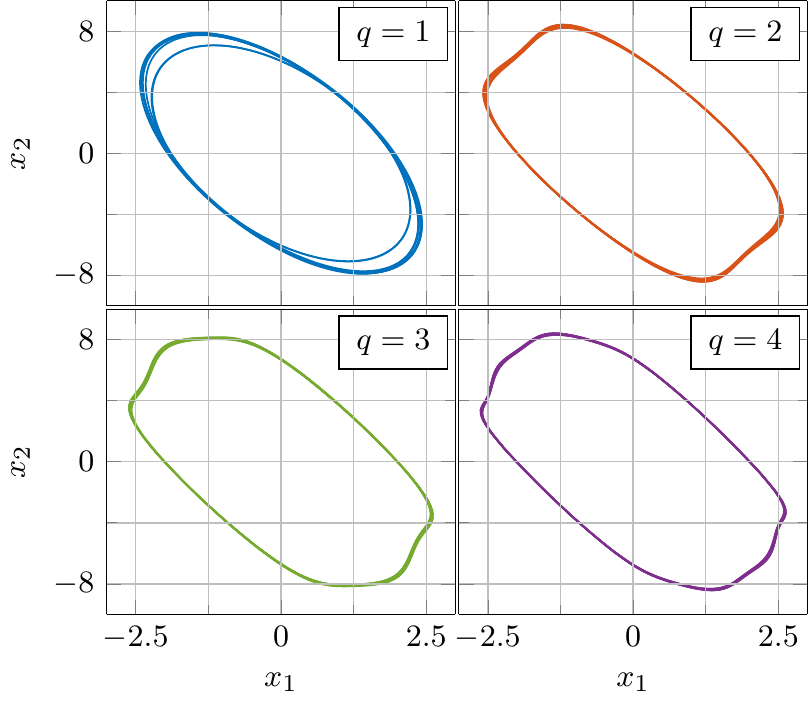}
	
	\caption{Randomized estimates $R_{N,q}$ of the set $\mathcal R_{250}$ for \eqref{eq:example}, \eqref{eq:mpc-suboptimal} with $N_1 = 1500$, $N_2 = 3000$, and varying degree $2q$.}
	\label{fig:scenarioopt-mpc}
\end{figure}

\begin{table}[h]
	\center
	
	\caption{Empirical accuracy of $R_{N,q}$ for \eqref{eq:example}, \eqref{eq:mpc-suboptimal} with $N_1 = 1500$, $N_2 = 3000$, and varying degree $2q$ (results in \num{0.01}).}
	\label{tab:scenarioopt-mpc}
	
	\sisetup{table-format=2.2, table-omit-exponent,  table-auto-round}
	\begin{tabular}{| c | S S | S S |}
		\hline
		 & \multicolumn{2}{S |}{{$\prob[M]{\mathcal R_p \setminus R_{N,q}$}}} & \multicolumn{2}{S |}{{$\prob[M]{R_{N,q} \setminus \mathcal R_p}$}} \\
		degree & {mean} & {max} & {mean} & {max} \\
		\hline
		$q = 1$ & 0.1185 & 0.1920 & 0.00044 & 0.00184 \\
		$q = 2$ & 0.0576 & 0.0726 & 0.00015 & 0.00047 \\
		$q = 3$ & 0.0327 & 0.0412 & 0.00020 & 0.00057 \\
		$q = 4$ & 0.0160 & 0.0227 & 0.00029 & 0.00056 \\
		\hline
	\end{tabular}
\end{table}

For each degree $q$, the average and worst-case empirical probabilities are detailed in Tab.~\ref{tab:scenarioopt-mpc}. It is shown that the rate of false-positive samples remains well below the accuracy of the inner estimation predicted by Theorem~\ref{thm:probest}. The rate of false-negative samples (relative to $\mathcal R_p$) appears to decrease with increasing polynomial degrees, even though $N_1$ was constant. It is important to note that the apparent fluctuation in the false-positive rate does not contradict our theoretical results but can be explained by reduced conservatism of the estimates $R_{N,q}$.

\section{Conclusion}
Data-driven stability estimates previously lacked probabilistic guarantees, had sample complexities scaling badly for larger systems, or provided only conservative approximations. Combining empirical optimization with ideas from converse Lyapunov theory, we have proposed an hierarchical, data-driven region-of-attraction estimation. While at each level the polynomial estimate is a probabilistic inner approximation of some desired accuracy, with sample complexity independent of the number of states or polynomial degree, we have proven that the estimates also converge in probability to be outer approximations. Thus, our approach is both accurate and non-conservative.

\appendices
%Appendixes, if needed, appear before the acknowledgment.
%\section*{Appendix}
\section{}
\label{app:invariance}
We propose a simple but efficient method for a lower bound on $\tilde r$ guaranteeing that $\mathscr B(\tilde r)$ is $\tilde p$-invariant. To that extent, write $\phi$ as
\begin{align}
	\phi = B_1 \circ \phi_1 \circ K_1 + \cdots + B_m \circ \phi_m \circ K_m
\end{align}
with $m \in \mathbb N$, where $\phi_1, \ldots, \phi_m: \mathbb R \to \mathbb R$ are scalar nonlinear functions and $\{ (B_i, K_i^{\mathrm T}) \in \mathbb R^n \times \mathbb R^n \}_{i=1}^m$ are pairs of linear operators. We make the following assumptions.

\begin{assumption}
	$\phi_1, \ldots, \phi_m$ are monotonic functions satisfying $| \phi_i(w) | \leq \phi_i(|w|)$ for all $w \in \mathbb R$ and $i \in \{1, \ldots, m\}$.
\end{assumption}

Let $\{x_k\}_{k \geq 0}$ be a solution of $\Pi(x_0)$ for an arbitrary initial condition $x_0 \in \mathbb R^n$ and define $w_{p,i} = K_i x_p$ for all $p \geq 0$ and $i \in \{1, \ldots, m\}$; 
we bound the state norm after $p \geq 0$ steps as
\begin{align}
	\label{eq:invariance-bound}
	\| x_p \|_2 \leq \| A^p \| \, r + \sum_{p'=0}^{p-1} \sum_{i=1}^m || A^{p-p'-1} B_i || \, \phi_i( | w_{p',i} | )
\end{align}
if $r = \| x_0 \|_2$,
where the $i$-th output after any $p' < p$ steps satisfies the recursive bound
\begin{align*}
	| w_{p',i} | \leq || K_i A^{p'} || \, r + \sum_{p''=0}^{p'-1} \sum_{i'=1}^m | K_j A^{p'-p''-1} B_i | \, \phi_i( | w_{p'',i'} | )
\end{align*}
and in particular, $| w_{0,i} | \leq || K_i || \, r$, for all $i \in \{1, \ldots, m\}$.
In other words, there exists a monotonic scalar bound
\begin{align}
	\| x_p \|_2 \leq F_p(r)
\end{align}
for all $x_0 \in \mathscr B(r)$.

A sequence of lower bounds $r_\iota < \tilde r$ for $\tilde p > 0$ can then be found as solutions to the family of optimization problems
\begin{align*}
	r_\iota = \max \{ r \geq 0 \, | \, F_p(r) \leq r - \iota, \; p \in \{ \tilde p, \ldots, 2\tilde p - 1 \} \} 
\end{align*}
for $\iota > 0$,
satisfying $\mathscr B(r_\iota) \subset \mathscr B(\tilde r)$.

\begin{proposition}
	\label{prop:invariance-bound}
	Let $\iota > 0$ be small; if $\| A^p \| < 1$ for all $p \geq \tilde p$ and $\phi_i(| w |) | \leq | w |^\alpha$ on $w \in \mathbb R$ for all $i \in \{ 1, \ldots, m \}$ and some $\alpha > 1$, then $r_\iota$ exists and satisfies $r_\iota > 0$.
\end{proposition}
\begin{proof}
	Replacing $\phi_i(\cdot)$ in \eqref{eq:invariance-bound}, we obtain
	\begin{align*}
		F_p(r) \leq \gamma r + \beta | r |^\alpha
	\end{align*}
	with $\gamma, \beta > 0$ for all $p \in \{ \tilde p, \ldots, 2\tilde p-1\}$; moreover, if $\| A^p \| < 1$ for all $p \geq \tilde p$, we can choose $\gamma < 1$. If $\alpha > 1$, then there exists $r^* > 0$ such that $\gamma r + \beta | r |^\alpha < r$ for all $r \in (0, r^*)$ and hence, if $\iota$ is sufficiently small, $F_p(r') \leq r' - \iota$ for some $r' \in (0, r^*)$. In other words, $r_\iota$ is the maximum of a nonempty set and $r_\iota \geq r'$, the desired result.
\end{proof}

The sufficient condition in Proposition~\ref{prop:invariance-bound} are satisfied for any stable matrix $A$ (compare \cite[Proposition~3.1]{Balint2006} and functions such as $w \mapsto w - \operatorname{sat}(w)$ or $w \mapsto w - \tanh(w)$, which are obtained in the analysis of optimization algorithms and neural network control, respectively.

\section{}
The following, auxiliary results are used to prove uniqueness of the polynomial approximation.

\subsection{Spanning the normal cone}
\label{app:conespan}
Let the convex set $K \subset \mathbb R^{n_\xi}$ satisfy
\begin{align*}
	K \subseteq \{ \xi \in \mathbb R^{n_\xi} \, | \, A \xi \leq b \}
\end{align*}
for some $A \in \mathbb R^{m \times n_\xi}$ and $b \in \mathbb R^m$. The normal cone of $K$ is defined as
\begin{align}
	N_K(\bar \xi) = \{ y \in \mathbb R^{n_\xi} \, | \, \text{$\langle y, \bar \xi - \xi \rangle \geq 0$ for all $\xi \in K$} \}
\end{align}
for any point $\bar \xi \in K$.

\begin{lemma}
	Let $\bar \xi \in K$ satisfy $A \bar \xi = b$ and $y \in \mathbb R^{n_\xi}$; there exist $\lambda \in \mathbb R_{\geq 0}^{m}$ such that $y = A^{\mathrm T} \lambda$ if and only if $y \in N_K(\bar \xi)$.
\end{lemma}
\begin{proof}
	Take $\lambda \geq 0$, then $\langle A^{\mathrm T} \lambda, \bar \xi - \xi \rangle = \langle \lambda, b - A \xi \rangle \geq 0$ for any $\xi \in \mathbb R^{n_\xi}$ satisfying $A \xi \leq b$ (i.e., for any $\xi \in K$) and hence, $A^{\mathrm T} \lambda \in N_K(\bar \xi)$. To see the opposite direction, assume that no such $\lambda \geq 0$ exists for $y \in \mathbb R^{n_\xi}$; by Farkas' lemma, then $A \nu \geq 0$ and $\langle y, \nu \rangle < 0$. Since $A (\bar \xi - \nu) \leq b$ we have that $y \not\in N_K(\bar \xi)$, the desired result.
\end{proof}

Note that, in the main part, the inequalities were flipped for both $K_2$ and $N^-_K(\bar \xi)$, thus recovering the result.

\subsection{Linear independence of monomials}
\label{app:vandermonde}
Let $x^{(1)}, \ldots, x^{(N)}$ be $N \in \mathbb N$ samples of a random variable $x \in \mathbb R^n$ and define the matrix
\begin{align*}
	S_q(x) = \begin{bmatrix} Z_q\{x^{(1)}\} & \cdots & Z_q\{x^{(N)}\} \end{bmatrix} \in \mathbb R^{n_q \times N}
\end{align*}
where $Z_q\{x\} \in \mathbb R^{n_q}$ denotes the vector of monomials of $x$ (starting with $x^0 \equiv 1$) up to degree $q \in \mathbb N$. We assume without loss of generality that $N = n_q$. The following statements are equivalent \cite[Theorem~4]{Olver2006}:
\begin{enumerate}
	\item The matrix $S_q(x)$ has full rank;
	\item The points $x^{(1)}, \ldots, x^{(N)}$ do not belong to a common algebraic hypersurface $\langle \psi, Z_q\{x\} \rangle = 0$ with $\psi \neq 0$;
	\item The interpolation $S_q(x)^{\mathrm T} \theta = v$ has a unique solution for any $v \in \mathbb R^N$.
\end{enumerate}
Note that these statements are equivalent to the null space of $S_q(x)$ being equal to $\{0\}$. They also hold independently of the actual elements or ordering of $Z_q\{x\}$. This allows to prove the following result by induction.

\begin{proposition}
	\label{prop:vandermonde}
	The matrix $S_q(x)$ has almost surely full rank.
\end{proposition}
\begin{proof}
	We show that any top-left $k \times k$ block $S^{[k]}$ of $S_q(x)$ with $k \in \{1, \ldots, n_q\}$ has full rank with probability one: Clearly, this is true for $S^{[1]} = 1$.
	
	Assume that $S^{[k]}$ has full rank; consequently, the top-left $(k+1) \times k$ block\footnote{This corresponds to the first $k+1$ elements of $Z_q\{x^{(i)}\}$ for the first $k$ samples $x^{(i)}$.} $S^{[k]}{}'$ of $S_q(x)$ has a null space spanned by $\psi \in \mathbb R^{k+1} - \{ 0 \}$. In other words, the samples $x^{(1)}, \ldots, x^{(k)}$ belong to the unique algebraic hypersurface $\langle (\psi, 0), Z_q\{x\} \rangle = 0$. %where $Z_q^{[k+1]}\{x\}$ denotes the first $k+1$ elements of $Z_q\{x\}$. 
	However, the probability that $x^{(k+1)}$ belongs to the same hypersurface is zero and hence, $S^{[k+1]}$ has full rank with probability one. Repeating this argument until $k = n_q$ %(and noting that $\prod_{k=1}^{n_q} 1 = 1$) 
	completes the proof.
\end{proof}

Since there are finitely many permutations of $N$ samples, we conclude that any combination of up to $n_q$ sampled vectors of monomials is linearly independent with probability one if $N > n_q$.

\section*{Acknowledgment}
The author remains thankful of Dominic Liao-McPherson for comments and discussions at various stages of the manuscript.

\bibliographystyle{IEEEtran}
\bibliography{../../../_bib/library}

%\begin{IEEEbiography}[{\includegraphics[width=1in,height=1.25in,clip,keepaspectratio]{cunis}}]{Torbjørn Cunis} (M'2022) received the M.Sc. degree in automation engineering from the RWTH Aachen University, Aachen, Germany, in 2016 and the doctoral degree in systems and control from ISAE-Supaéro, Toulouse, France, in 2019.
%
%He was a Researcher at ONERA -- The French Aerospace Lab from 2016 to 2019 and a Research Fellow of the University of Michigan from 2019 to 2021. Since 2021, he has been a Lecturer at the University of Stuttgart Institute of Flight Mechanics, Stuttgart, Germany. His research is concerned with the analysis and verification of nonlinear system dynamics, specifically for autonomous vehicles and aircraft, optimization algorithms, and artificial intelligence.
%
%Dr. Cunis is a member of AIAA and VDI, an adjunct at the University of Michigan Aerospace Department, and a fellow of the Young ZiF at the Centre for Interdisciplinary Research, University of Bielefeld.
%\end{IEEEbiography}

\end{document}

%% file: inner.tex
\begin{algorithmic}[1]
	\REQUIRE $N$ iid. instances of $\breve x \in \rdvar {\mathcal R_p}$ and $\hat x \in \rdvar {(\mathcal D \setminus \mathcal R_p)}$
	\STATE $v^{(i)} := V_p(\breve x^{(i)})$ and $z_1^{(i)} := Z^2_q\{ \breve x^{(i)} \}$ for $i = 1, \ldots, N$
	\STATE $z_2^{(j)} := Z^2_q\{ \hat x^{(j)} \}$ for $j = 1, \ldots, N$
	\STATE find $(\eta, \Theta) \in \mathbb R \times \mathbb S^+_{n_q}$ minimizing $\eta$ subject to
		\begin{align*}
			\left.
		\begin{alignedat}{2}
			\langle z_1^{(i)}, \Theta \rangle + \eta &\geq & {} &v^{(i)} \\
			\eta - \langle z_1^{(i)}, \Theta \rangle &\geq & -{} &v^{(i)} \\
			\langle z_2^{(j)}, \Theta \rangle + \eta &\geq & {} &c_p
		\end{alignedat}
			\right\} \quad 
		\begin{aligned}
			&\text{for $i = 1, \ldots, N$} \\
			&\text{for $j = 1, \ldots, N$}
		\end{aligned}
		\end{align*}
	\ENSURE optimal solution $(\eta^N, \Theta_N)$
\end{algorithmic}

%% file: outer.tex
\begin{algorithmic}[1]
	\REQUIRE $N$ iid. instances of $\bar x \in \rdvar {(\mathcal D - \mathcal R_p)}$ and $\Theta \in \mathbb S_{n_q}^+$
	\STATE $\tau^{(j)} := \langle \Theta, Z^2_q\{ \bar x^{(j)} \} \rangle$ for $j = 1, \ldots, N$
	\STATE find $c \in \mathbb R$ maximizing $c$ subject to
		\begin{align*}
			c \leq \tau^{(j)} \quad \text{for $j = 1, \ldots, N$}
		\end{align*}
	\ENSURE optimal value $c^N$
\end{algorithmic}